\documentclass[11pt]{article}
\usepackage[left=1in,top=1in,right=1in,bottom=1in,nohead]{geometry}
\usepackage{mathrsfs}
\usepackage[OT2, T1]{fontenc}
\usepackage{url}
\usepackage{amsmath}
\usepackage{array}
\usepackage{graphicx}
\usepackage{amsfonts}
\usepackage{amssymb}
\usepackage{amstext}
\usepackage{amsthm}
\usepackage{hyperref}
\usepackage{colonequals}
\usepackage{enumitem}
\usepackage[alphabetic,lite]{amsrefs}
\usepackage{cleveref}
\usepackage[all,cmtip]{xy}
\usepackage{fullpage}

\numberwithin{equation}{subsection}

\newtheorem{theorem}{Theorem}[section]
\newtheorem{lemma}[theorem]{Lemma}
\newtheorem{prop}[theorem]{Proposition}
\newtheorem{corollary}[theorem]{Corollary}
\newtheorem{thm}{Theorem}

\newtheorem{ques}[theorem]{Question}

\theoremstyle{definition}

\theoremstyle{remark}

\newcommand{\A}{\mathbb{A}}

\newcommand{\bC}{\mathbb{C}}

\newcommand{\G}{\mathbb{G}}
\newcommand{\Hh}{\mathbb{H}}
\newcommand{\bH}{\mathbb{H}}

\newcommand{\Pp}{\mathbb{P}}
\newcommand{\Q}{\mathbb{Q}}
\newcommand{\Qbar}{\overline{\Q}}
\newcommand{\R}{\mathbb{R}}
\newcommand{\Ss}{\mathbb{S}}

\newcommand{\Z}{\mathbb{Z}}

\newcommand{\cA}{\mathcal{A}}

\newcommand{\cF}{\mathcal{F}}

\newcommand{\cH}{\mathcal{H}}

\newcommand{\fsp}{\mathfrak{s}\mathfrak{p}}

\newcommand{\alg}{{\mathrm{alg}}}

\newcommand{\ra}{\rightarrow}
\newcommand{\im}{\mathrm{im}}

\DeclareMathOperator{\GL}{GL}

\DeclareMathOperator{\GSp}{GSp}
\DeclareMathOperator{\Sp}{Sp}

\DeclareMathOperator{\tor}{tor}

\DeclareMathOperator{\End}{End}
\DeclareMathOperator{\Hom}{Hom}
\DeclareMathOperator{\Aut}{Aut}

\DeclareMathOperator{\Res}{Res}

\DeclareMathOperator{\ad}{ad}

\usepackage[usenames,dvipsnames]{color}  

\author{Jacob Tsimerman}

\date{}

\begin{document}
\title{Abelian Varieties are not quotients of low-dimension Jacobians}
\maketitle
\begin{abstract}
    We prove that for any two integers $g\geq 4$ and $g'\leq 2g-1$, there exist abelian varieties over $\Qbar$ which are not quotients of a Jacobian of dimension $g'$. Our method in fact proves that most abelian varieties satisfy this property, when counting by height relative to a fixed finite map to projective space.


\end{abstract}

\section{Introduction}

Jacobians are a subset of abelian varieties, and much work has been devoted to figuring out `how large' that subset is. An elementary computation shows that for $g\geq 4$ the generic abelian varieties is not a Jacobian. Over 
$\bC$, this immediately implies that the generic abelian variety is not isogenous to a Jacobian, either. This can be seen in many ways, and essentially reduces to the fact that $\bC$ is an uncountable field, and therefore the union of countably many subvarieties cannot be equal to an ambient variety of higher dimension.

Katz and Oort conjectured that the same property holds over $\Qbar$. Over a countable field this problem becomes much more difficult, and was solved unconditionally in \cite{COJacobians,tsimjacobians} by relying on abelian varieties with complex multiplication, and appealing to the Andr\'e-Oort conjecture. Later, a different proof was given in \cite{MZjacobians}, by using the Pila-Zannier method and isogeny estimates of Masser-W\"ustholz\cite{MWisogeny}. 

A related fact in the positive direction is that every abelian variety $A$ is a quotient of a Jacobian. Indeed, one may find a smooth curve inside of $A$ not contained in a coset of a proper abelian subvariety and consider its Jacobian. One may ask for the lowest such number:

\begin{ques}\label{ques:lowestj}
    Let $g\geq 4$, and let $A$ be a generic abelian variety over an algebraically closed field $k$. What is the lowest integer $g'$ such that there exists a Jacobian $J$ of dimension $g'$ and a surjection $\phi:J\ra A$?
\end{ques}

Question \ref{ques:lowestj} is extremely difficult even for $k=\bC$. Comparing dimensions shows that $g'\geq \frac{g^2+g}{6}+1$, but this is far from optimal. In \cite{pirolaabelian} it is proven that $g'\geq \frac{g(g-1)}{2}+1$, with equality iff $g\in\{4,5\}$. For an upper bound there does not appear to be a better method known than taking a projective embedding and intersecting with hyperplanes to produce a curve, which gives an upper bound which looks like a power of $g!$. There is therefore a huge chasm and to the author's knowledge there is not even a conjecture as to what the truth should be. Relatedly, see work of Voisin\cite{voisinchowring} and Pirola\cite{pirolahyperelliptic} on giving a lower bound for the \emph{gonality} of a curve in a generic abelian variety.

The goal of this paper is to make progress on question \ref{ques:lowestj} for $k=\Qbar$:

\begin{thm}\label{thm:mainTorelli}
    Let $g\geq 4$. Then there exist abelian varieties over $\Qbar$ which are not quotients of a Jacobian of a curve of genus $\leq 2g-1$.
\end{thm}

In fact, for our proof there is nothing special about the Torelli locus. We prove the following Theorem, where $\cA_g$ is the moduli space of Principally Polarized abelian varieties:
\begin{thm}\label{thm:main}
    Let $0\leq h<g$ be integers. Suppose that $V\subset \cA_{g+h}$ is a subvariety such that over $\bC$, the generic $g$-dimensional abelian variety is not a quotient of an abelian variety parameterized by $V$.

    Then there exist $g$-dimensional abelian varieties over $\Qbar$ which are not a quotient of an abelian variety parameterized by $V$. 
    
\end{thm}

In fact we shall show in Theorem \ref{thm:mainsstrong} that in a precise sense (analogous to what Masser, Zannier\cite{MZjacobians} and Zywina \cite{Zywina} do) most abelian varieties are not quotients of abelian varieties parameterized by $V$.

\subsection{Summary of the Proof}

Our proof follows the Pila-Zannier method, as employed by Masser and Zannier. In fact, in the case of $h=0$ our proof essentially specializes to theirs\footnote{We find it convenient to use slightly different counting functions (with rationals instead of integers more in line with \cite{Zywina}), but the essence of our argument reduces to exactly what they do.}. The main ideas are as follows:

\begin{enumerate}
    \item The main difficulty - and the main technical contribution of this paper - is to demonstrate `Large Galois Orbits'. Concretely, suppose that $A\in\cA_g(K)$ is an abelian variety defined over a number field $K$ and $B\in V(\Qbar)$ such that $A$ is a quotient of $B$. Then there is a degree $N$ isogeny $\phi$ from  $A\times C$ to $B$ for some $h$-dimensional abelian variety $C$. We wish to show that if there is one such $B$ then there are many - presuming $A$ is generic in some way (see Lemma \ref{prop:openimage} for the specifics). In particular, we wish to show that $B$ has a large field of definition so that we may consider all of the Galois conjugates of $B$.
    
    The difficulty is that one has no a-priori control over either the height or field of definition of $C$. We get around this by using the fact that $g>h$, and specifically that there is no embedding $\GSp_{2g}$ into $\GSp_{2h}$. Morally, this means that the Galois action on the Tate module of $C$ cannot `interact' with the action on the Tate module of $A$. Concretely, we let $K\subset A\times C$ be the kernel of $\phi$. Then we show (after some reductions) that for generic $A$, we have $K_A:=K\cap A\times\{0_C\}$ is large compared to $N$. This means that the abelian sub-variety in $B$ which is isogenous to $A$ has large field of definition, and hence $B$ must as well.

    \item Armed with this we use the Pila-Wilkie Theorem to obtain an unlikely intersection in $V$. We then use powerful transcendence theorems (Ax-Schanuel and Geometric Zilber-Pink, see \S\ref{subsection:transcendence}) to reduce to the case where $V$ is contained in a weakly special subvariety. 

    \item The last difficult piece is handling the case where $V$ is contained in a weakly special subvariety. We handle this in a somewhat unsatisfying (but very clean) ad-hoc method. Specifically, we show that if a weakly special subvariety contains a non-simple point with a Hodge-generic factor,  then the generic point over it is non-simple\footnote{Note that this is false without the Hodge-Generic assumption. For example, consider quaternionic curves in $\cA_2$, which contain non-simple CM points.}. This allows us to carry out an induction argument to finish. 
    
\end{enumerate}

In actuality we reverse steps 2 and 3, but the above outline is the heart of the proof. We remark on the limitation of $h<g$ for our method:

The reason that this is a problem is the step where we establish large Galois orbits. Specifically, the issue is a possible counterexample to Frey-Mazur: Let $A$ be a $g$-dimensional abelian variety and suppose that $C$ is another 
one such that $A[p]\cong C[p]$ as Galois modules. Then we may quotient out $A\times C$ by the graph of this isomorphism to obtain a $2g$-dimensional abelian variety defined over $\Q$\footnote{To get the polarizations working we should insist $p\equiv1\mod 4$ and add a quadratic twist so that the graph is a Lagrangian, but this is a minor technical point.}. We are unable to rule this out from happening for arbitrarily large primes, and hence we cannot get a handle on $g=h$. It seems to be a very difficult problem to rule such a thing out, even for a single $A$!

If one were able to circumvent this issue and establish large Galois orbits, then the rest of the proof would go through unchanged.

\subsection{Outline of the Paper}

In \S2 we recall background about abelian varieties, the corresponding  Galois representations, functional transcendence, and the Pila-Wilkie method. In \S3 we prove our `large Galois orbits' result. In \S4 we prove a result about weakly special subvarieties which enables us to reduce to the case of $V$ not being contained in a proper weakly special. Finally, in \S5 we combine everything to prove the main theorem.

\subsection{Acknowledgements}
It is a pleasure to thank Jonathan Pila and Ananth Shankar for enlightening discussions on this problem.

\section{Background}
\subsection{Polynomial boundedness}

We write $A\prec_C B$ or $B\succ_C A$ or non-negative functions $A,B$ depending on $C$ and possibly other variables, if there are constants $c_1,c_2,c_3>0$ which depend only on $C$, such that $A<c_1B^{c_2}+c_3$. If $A\prec_g B$ and $B\prec_g A$ we write $B\sim_g A$. If $A\prec_g B$ we say that $A$ is \emph{polynomially bounded by $B$}.

\subsection{Abelian Varieties}

We fix some notation for abelian varieties:

\begin{itemize}
    \item A \emph{quasi-morphism} between abelian varieties $A,B$ is an element of $\Hom(A,B)\otimes\Q$. 
    \item A \emph{quasi-isogeny} is a quasi-morphism $\phi$ such that for any integer $n$ for which $n\phi$ is a morphism, $n\phi$ has finite kernel.
    \item For a complex abelian variety $\bC$, we denote $L_A:=H_1(A,\Z), V_A:=H_1(A,\Q), T_A:=V_A/L_A$. Note that $T_A\cong A_{\tor}(\bC).$
    \item We denote the Tate module by $T_{A,\ell}:=\varinjlim A[\ell^n].$
\end{itemize}

\subsection{Principal Polarisations and Isogenies}\label{subsection:polariztions}

Recall that a polarization of a complex abelian variety $A$ is an isogeny $\phi:A\ra A^\vee$ which is symmetric, and for which the pullback of the Poincare bundle along the associated graph is ample. Given a polarization, we get an associated symplectic pairing $Q_{A,\phi}$ known as the \textit{Weil pairing} on $L_A$, and also $Q_{A,\phi,n}$ on $A[n]$ for each positive integer $n$. We say that $\phi$ is principal if the degree of $\phi$ is $1$. 

We say that two polarizations $\phi,\psi$ are \emph{equivalent} if there are positive integers $m,n$ such that $m\phi=n\psi$. Given another abelian variety $B$ and an isogeny $f:B\ra A$
we define a polarization on $B$ by setting $f^*\phi:= f^\vee\circ\phi\circ f$. Given two polarized abelian varieties $(A,\phi),(B,\psi)$ we say that an isogeny $f:A\ra B$ is a \emph{polarized isogeny} from $(A,\phi)$ to $(B,\psi)$ if $f^*\psi$ is equivalent to $\phi$. 

If $\phi,\phi'$ are two polarizations on $A$ then we obtain a quasi-isogeny $f:=\phi^{-1}\circ \phi':A\ra A$ such that the induced map on $V_A$ respects the pairing $Q_{A,\phi}$. In particular, 
$$Q_{A,\phi}(v_1,fv_2)=Q_{A,\phi}(fv_2,v_1).$$

\subsection{Shimura Varieties}

See \cite{milnemoduli} for further background on this section.

Let $G$ be a reductive group over $\Q$, $\Ss=\Res_{\bC/\R}\G_m $ the Deligne torus, and $X$ be a conjugacy class of 
homomorphisms $f:\Ss\ra G_\R$ satisfying the Shimura axioms, and $K=\prod_p K_p\subset G(\A_f)$ be a neat compact subgroup. Let $E(G,X)$ be the reflex field, and $E$ be a field over which $G$ splits. Associated to this data we get an algebraic variety $S_K(G,X)$ defined over $E(G,X)$ called a "Shimura Variety" whose complex points can be identified with $G(\Q)\backslash\left(X\times G(\A_f)\right)/K$.

We shall primarily be concerned with $\cA_g$ which is the Shimura variety corresponding to $G=\GSp_{2g}$. This naturally gives $\cA_g$ the structure of the (coarse) moduli space of principally polarized abelian varieties of dimension $g$. In this case $E(G,X)=\Q$ so we consider $\cA_g$ as a $\Q$-variety.

\subsubsection{Weakly Special Varieties}

We follow \cite{UY} in this section:

Suppose $S_K(G,X)$ is a Shimura variety for which we have:

\begin{itemize}
    \item A Shimura subvariety $S_{K_H}(H,X_H)$
    \item Shimura varieties $(H_i,X_i,K_i)$ for $i=1,2$  and a splitting $H^{\ad}\cong H_1\times H_2$ inducing a splitting $X_H\cong X_{H_2}\times X_{H_2}$.
\end{itemize}

Then we say that the image of $X_1^+\times \{y\}$ is a \emph{weakly special} subvariety of $S_K(G,X)$ for any connected component $X_1^+$ of $X_1$ and any point $y\in X_2$. We say that a closed analytic subvariety $S\subset X_G$ is weakly special if it is a connected component of the preimage of a weakly-special variety. 
\subsection{Functional Transcendence Theorems}\label{subsection:transcendence}

We let $\cH_g:=\{Z\in M_g(\bC): Z=Z^t, \im Z>0\}$ denote the Siegel upper-half space, so that there is a natural uniformization map $\pi_g:\cH_g\ra \cA_g(\bC)$. Then we have a semi-algebraic structure on $\cH_g$ induced from the algebraic structure of $M_g(\bC)$. It is easy to show that weakly special varieties $V$ are bi-algebraic for the map $\pi_g$, in the sense that all of the irreducible components of $\pi_g^{-1}(V)$ are semi-algebraic. The Ax-Schanuel Theorem \cite{ASAg} shows that this is the only algebraic information that is preserved by $\pi_g$:

\begin{thm}\label{thm:AS}
    Let $V\subset \cH_g\times \cA_g$ be an irreducible algebraic variety, and let $\Gamma\subset \cH_g\times \cA_g$ be the graph of the map $\pi_g$. Let $U\subset V\cap\Gamma$ be an irreducible component and assume that $\dim U>\dim V+\dim\Gamma - \dim\cA_g$. Then $\pi_g(U)$ is contained in a proper weakly-special subvariety of $\cA_g$.
\end{thm}

We shall also require another theorem, which takes into account unlikely intersections with different weakly-special varieties at once. Namely, let $V\subset \cA_g$ be a variety, and we assume it is not contained in any proper weakly-special subvariety. We say that an irreducible subvariety $X\subset V$ is \emph{weakly-atypical} if there is a weakly special subvariety $W\subset\cH_g$ containing $X$such that $$\dim X > \dim W + \dim V - \dim\cA_g.$$ We then have the following:

\begin{thm}\label{thm:GZP}
$V\subset \cA_g$ be a variety which is not contained in any proper weakly-special subvariety. Then there is a proper subvariety $Z\subset V$ which contains all positive-dimensional weakly-atypical subvarieties of $V$, such that for any irreducible component $Z^o\subset Z$ we have
\begin{enumerate}
    \item $Z^o$ is weakly atypical, or
    \item $Z^o$ is contained in a proper special subvariety of $\cA_g$.
\end{enumerate}

\end{thm}

\begin{proof}
    Note that $V$ itself is not weakly atypical, by our assumptions. The result now follows directly from \cite[Thm 6.1]{baldiklinglerullmo} by noting that the adjoint group of $\GSp_{2g}$ is simple, and therefore for case (ii) in Theorem 6.1 of that paper to occur $Z^o$ must be contained in a proper special subvariety.
\end{proof}

The above is named the Geometric Zilber-Pink conjecture for the weakly special atypical locus by \cite{baldiklinglerullmo}, and is proven there in the more general context of hodge variations.

\subsection{Pila-Wilkie Theorem}

Recall that the height of a rational point $\frac ab$ with $\gcd(a,b)=1$ is $H(\frac ab):=\max(|a|,|b|)$. For a point $q=(q_1,\dots,q_n)\in\Q^n$ we define
$H(q):=\max_i H(q_i)$. Finally, for any subset $X\subset\R^n$ we define $$N(X,T):=\#\{q\in X\cap \Q^n\mid H(q)\leq T\}.$$

Given a set $X\subset \R^n$, we set $X^{\alg}$ to be the union of all connected, positive dimensional, semi-algebraic subsets of $X$. Then we have the following  Theorem of Pila--Wilkie:

\begin{thm}[Pila--Wilkie {\cite[Thm 1.8]{PW}}]\label{pointcount}
For a set $X\subset\R^n$ definable in an o-minimal structure, and any $\epsilon>0$, we have $$N(X-X^{\alg},T) =  T^{o(1)}.$$
\end{thm}

In fact, we shall need the slightly stronger version, valid in families.

\begin{thm}[Pila--Wilkie {\cite[Thm 3.6]{PAOmodular}}]\label{pointcount2}
For a set $X\times Y\subset\R^n$ definable in an o-minimal structure, we have that 
$$N(X_y-X_y^{\alg},T)=T^{o(1)}$$ uniformly in $y\in Y$.

\end{thm}

We shall only use the o-minimal structure $\R_{an,\exp}$ in which all the sets we work with are defined. For background on o-minimality see \cite{Dries}.

\subsection{Height Estimates}
We use \cite{PTsurfaces} for background. 

We work in the Siegel Upper-Half space $\Hh_g$ of symmetric $g\times g$ complex matrices with positive definite imaginary component. 

for a matrix $Z=X+iY\in\Hh_g$ we define its height to be 
$$H(Z):=\max\big(|z_{i,j},\frac1{|Y|}\big).$$

We set $\cF_g\subset\Hh_g$ to be the standard fundamental domain\cite{siegelsympleticbook}. The following is \cite[Lemma 3.2]{PTsurfaces}

\begin{lemma}\label{lem:fundheightbound}
    Let $Z\in\bH_g$ and $\gamma\in\GSp_{2g}(\Z)$ such that $\gamma Z\in \cF_g$. Then the co-ordinates of $\gamma$ and the Height of $\gamma Z$ are $\prec_g H(Z)$.
\end{lemma}

As a result, we have the following simple consequence:

\begin{corollary} \label{cor: isogenyheightbound}
    Let $Z,W\in \cF_g$ be such that there is a degree $N$ polarized-isogeny $\phi:A_Z\ra A_W$. Then $H(W)\prec_g H(Z)+N$. Moreover, there is an element $M\in \GSp_{2g}(\Q)$ such that $M Z=W$ and $H(M)\prec_g N$.
\end{corollary}

\begin{proof}
    The existence of $\phi$ means there is an element $M\in \GSp_{2g}(\Q)\cap M_{2g}(\Z)$ with determinant $N$ such that $MZ=W$. It is well-known that these break up into finitely many right $\GSp_{2g}(\Z)$ cosets, and we may write $M=\gamma\gamma_0$ for $\gamma\in\GSp_{2g}(\Z)$ and $\gamma_0$ having entries bounded by $N^{O_g(1)}$. The bound on $\gamma_0$ follows for $N$ a prime power by the explicit coset representatives computed in \cite{explicitheckecosets}, and in general by multiplying and using that $(2g)^{\omega(N)}=N^{o_g(1)}$. 
This implies that $H(\gamma_0 Z)\prec (H(Z)+N)$, and so the bound on $H(W)$ follows by Lemma \ref{lem:fundheightbound}. 

Finally it remains to prove the bound on $H(M)$. This follows directly from \cite[Thm 1.1]{orr} for the group $G=\GSp_{2g}$ with its canonical representation. 

    \end{proof}

\subsection{Galois Representations}

Let $A$ be an abelian variety over a number field $K$. We get an induced Galois representation $\rho_{A,\ell}:G_K\ra\GL(T_{A,\ell})$ for each prime $\ell$, and therefore $\rho_A:G_K\ra\GL(T_A)$. This must preserve every polarization and endomorphism defined over $K$. For each $\ell$, we obtain an $\ell-adic$ Mumford-Tate Group $MT_{A,\ell}$ by taking the Zariski closure of the image of $\rho_{A,\ell}$. Note that the $\ell$-adic Mumford-Tate group depends on the number field but its identity component (as an algebraic group) does not. 

Given a polarized abelian variety $(A,\phi)$, we say that $A$ is \emph{Galois Generic} if the image of $\rho_A$ is open in $\GSp(T_A,\phi)$. We remark that this is independent of $\phi$ since this property implies $\End(A)=\Z$ and thus that $A$ has a single polarization up to equivalency.

We have the following result (see \cite[Remark17.3.1]{BLAV})
 \begin{prop}\label{prop:abshodge}
   The identity component of $MT_{A,\ell}$ is contained in the extension of scalars of the ordinary Mumford-Tate group $MT_A$ from $\Q$ to $\Q_{\ell}$.
   \end{prop}

\subsection{Counting by Height}\label{subsection:counting}

The proof of our Theorem requires a notion of `most' abelian varieties, so we need a way to count a subset of points  $\cA_g$ which are infinite, easy to get at, and have generic behaviour. To do this we borrow the mechanism used by Masser and Zannier\cite{MZjacobians}. We fix an open set $U\subset \cA_g$ and a quasi-finite dominant map $\phi:U\ra\Pp^{\dim\cA_g}$ defined over $\Q$. We may then count $\phi$-rational points by defining
$$S_{\phi}(T):=\{t\in \cA_g(\Qbar)\mid \phi(t)\in U(\Q), H(\phi(t))\leq T\} $$ for each $T>0$, where $H$ is the usual height function on projective space. Note that the abelian varieties in $S_{\phi}(T)$ are defined over number fields of bounded degree.

We call such a $\phi$ a \emph{counting function} on $\cA_g$, and we say that a property holds \emph{$\phi$-generically} on $\cA_g$ if the proportion of points in $S_{\phi}(T)$ that it holds for tends to  1 as $T\ra\infty$.

\section{Galois Orbits}

Let $A$ be an abelian variety of dimension $g$ defined over a number field $K$. Let $\rho_A:G_K\ra \GL_{2g}(\hat{\Z})$ be the (isomorphism class of the ) Galois representation associated to an abelian variety. Of course, if it is principally polarized then the image is contained in $\GSp_{2g}(\hat\Z)$. If the image of $\rho_A$ has finite index in $\GSp_{2g}(\hat\Z)$ then we call its index the \emph{index of $A$}. Note that this implies that $A$ is simple.

We recall the following theorem:

\begin{prop}\label{prop:openimage}\cite[Theorem 1.1]{Zywina}

Fix a counting function $\phi$ as in \ref{subsection:counting}. Then $\phi$-generically, for an abelian variety $A\in\cA_g$ we have that the index $[\GSp_{2g}(\hat\Z):\im\rho_A]$
is finite of size $O_g(1)$. It follows in particular that $A$ is Hodge-generic.

\end{prop}

\begin{proof}

This follows from Theorem 1.1 of \cite{Zywina} in exactly the same way as Theorem 1.2 of \cite{Zywina} does, by applying Theorem 1.1 to the restriction of scalars after passing to a Galois cover. 
    
\end{proof}

The purpose of this sections is to prove a large Galois orbits result:

\begin{thm}\label{thm:largegalois}
    Let $A$ be a ppav of dimension $g$ over number field $K$ of degree $O_g(1)$ over $\Q$ such that $[\GSp_{2g}(\hat\Z):\im\rho_A]=O_g(1)$. Fix $0\leq h< g$. Let $B$ be a $\Qbar$ ppav of dimension $g+h$ such that the lowest degree isogeny to $B$ from an abelian variety of the form $A\times C$, for a ppav $C$ of dimension $h$, is $N$. 
    
    Then the degree of the field of definition of $B$ is $\succ_g N$. Moreover, $B$ has a \textbf{polarized isogeny} from $A\times C'$ for a (possibly different) ppav $C'$ of degree $\prec_g N$.
\end{thm}

\begin{proof}

    Note that since $g>h$ and $A$ is simple that $\Hom(A,C)=\Hom(C,A)=0$, and thus $\End(A\times C)\cong \Z\oplus \End(C)$. We thus record elements of $\End(A\times C)$ as pairs $(m,r)$ for $r\in\End(C)$. 
    
    We let $L_B\subset V_A\times V_C$ be the Lattice corresponding to $H_1(B,\Z)$ under the isogeny to $B$. We identify $V_B$ with $V_A\times V_C$. Note that $[L_B:L_A\times L_C]=N$. Now there is a natural form symplectic form $Q$ on $V_A\times V_C$ corresponding to the principal polarizations $\phi_A,\phi_C$  on $A,C$,  and this gives us an isomorphism
    $B^\vee(\bC)\cong V_{A,C}/L_B^*$ where $L_B^*$ denotes the dual of $L_B$ under $Q$. Corresponding to the principal polarization $\phi_B$ of $B$, there is therefore an endomorphism $f=(m,r)\in End(A\times C)$ such that $fL_B=L_B^*$, and by \ref{subsection:polariztions} we have that $f$ respects $Q$.

    \textbf{\textrm{Step 1: Proving} $m\succ_g N$ }
    
    \medskip
    Note first that $m^{g}(\deg r)^{\frac 12}= N.$ We have the lattices $$\pi_A(L_B)\supset L_B\cap V_A\supset L_A, \pi_C(L_B)\supset L_B\cap V_C\supset L_C$$ and by duality of $L_B$ and $fL_B$, we have the relations
    $$m\pi_A(L_B) = (V_A\cap L_B)^*, r(V_C\cap L_B)=\pi_C(L_B)^*,  r\pi_C(L_B)=(V_C\cap L_B)^*$$.

    Note that $\pi_C(L_B)/(V_C\cap L_B)\cong \pi_A(L_B)/(V_A\cap L_B)$ and thus is of size $d\leq m^{2\dim A}$.     Dualizing we get
    $$d= [(V_C\cap L_B)^*: \pi_C(L_B)^*] =[(V_C\cap L_B)^*: r(V_C\cap L_B)] $$

    Let $C':=V_C/(V_C\cap L_B)$. Then $C'^\vee \cong V_C/\pi_C(V_C\cap L_B)^*$ and $r:C'\ra C'^\vee$ is a polarization on $C'$ of degree $d$. As a result we may find an polarized isogeny $C''\ra C'$ of degree $\leq d$ where $C''$ is principally polarized. Then the natural isogeny $ A\times C''\ra B$ is of degree at most $dm^{2g}\leq m^{4g}$. By virtue of the minimality of $N$, we conclude that $\deg r\leq m^{6g}$ and thus that $m\geq N^{\frac {1}{4g}}$. 

    \textbf{\textrm{Step 2:} Getting a polarized isogeny }
    
    \medskip

    Replacing $C$ by the $C'$ above (at the cost of replacing $N$ by $N'\prec_g N$) we may assume that $r$ is multiplication by some positive integer $m'$.

    We are now in the setup where $m\prec_g N$ and $r$ is multiplication by a positive integer $m'$. We now observe that we may replace $(C,\phi_C)$ by $(C',\phi_C\cdot\frac1k)$ by changing $L_C$ to $L'_C$
    where $L'_C$ mod $k$ is maximally isotropic for the weil-pairing. This has the effect of changing $m'\ra m'k$. We may also replace $L_C$ by $\frac{1}{a}L_C$ and $\phi_C$ by $k^2\phi_C$ which sends $m'\ra m'/a^2$. Thus by changing $C$ and the polarization in this way we may reduce to $m'=m$ by first scaling by $k=mm'$ and then dividing by $a=m'$. This gives a polarized isogeny $A\times C\ra B$ of size $\prec_g N$. 

    \textbf{\textrm{Step 3:} Large Galois Orbit }
    
    \medskip

    We finally prove that the $B$ is defined over a large field of definition. To prove this, we note that $B$ contains a unique abelian subvariety $D$ which is isogenous to $A$. It suffices to show therefore that $D$ has a large field of definition. We now consider again the minimal isogeny $f=(m,r):A\times C\ra B $, and write let $G:=L_B/(L_A\times L_C)$. Note that there is a sequence
    $$f^{-1}(L_A\times L_C)\supset L_B\supset L_A\times L_C$$ whose successive quotients are $G$ and $G^{\vee}$, and note that the latter is isomorphic to $G$ as an abelian group. Moreover, the total quotient is $A[m]\times C[r]$. It follows that $\#G[m]^2\geq \#A[m]\cdot\#C[(r,m)]$. 
    But now $G^\vee[m]\subset A[m]\times C[(r,m)]$  and therefore
    $$\#(G^\vee\cap A[m])\geq \big(\frac{\#A[m]}{\#C[(r,m)]}\big)^{\frac12}\geq m$$ where the last inequality follows from $g\geq h+1$. It follows that $G^\vee\cap A$ contains an element of order $e\geq m^{\frac1{2g}}\succ_g N$. 
    
    Now, by minimality of $N$ we see that $G^\vee\cap A[m]$ does not contain any $A[\ell]$ for a prime $\ell$. Thus we see that $\#(G^\vee\cap A[e])\leq e^{2g-1}$. On the other hand, the orbit under $\GSp_{2g}(\hat\Z)$ of $G^\vee\cap A[e]$ contains all of $A[e]$ which is of size $e^{2g}$. It follows that there are at least $e$ such orbits under $\GSp_{2g}(\hat\Z)$. It follows by the assumption on $\rho_A$ that $G^\vee\cap A[e]$ has $\succ_g N$ orbits under the action of Galois.
    The claim now follows from the isomorphism $D\cong A/(G^\vee\cap A)$ together with the simplicity of $A$.

\end{proof}

\section{Geometric Arguments: Weakly Special Subvarieties}

We shall require the following:

\begin{prop}\label{prop:specialssplit}
    Let $B$ be a ppav of dimension $g+h$ which has a quotient $A$ which is Galois generic of dimension $g$. Let $S$ be a proper special subvariety of $\cA_g$ which contains $[B]$. Then the universal abelian variety over $S$ is not geometrically simple.
\end{prop}

\begin{proof}

Let $G\subset \GSp_{2(g+h),\Q}$ be the subgroup corresponding to $S$. Now the $\ell$-adic derived Mumford tate group $MT^{\textrm{der}}_{\ell}(B)$ surjects onto $MT^{\textrm{der}}_{\ell}(A)$ and thus onto $\Sp_{2g,\Q_\ell}$. Since 
$\Sp_{2g}$ is a simple algebraic group, it follows by Goursat's Lemma that $MT_{\ell}(A)$ contains $\Sp_{2g}$ and hence by Proposition \ref{prop:abshodge} that $MT(A)$ does as well, and hence so does $G$. By Lemma \ref{lem:grp} it follows that $G$ preserves a non-trivial $\Q$-symplectic splitting, and therefore the universal abelian variety over $S$ is not geometrically simple.

\end{proof}

The following Lemma is almost certainly not original but we could not find it in the literature:

\begin{lemma}\label{lem:grp}

Let $V,W$ be symplectic spaces over $\bC$ and let $G\subset\GSp(V\oplus W)$ be a connected reductive group which contains $\Sp(V)$. Then either $G$ contains $\Sp(V\oplus W)$ or else there is a proper symplectic subspace $W'\subset W$ such that $G$ preserves $V\oplus W'$. Moreover, if $V,W,G$ are defined over a field $F\subset\bC$ then $W'$ can be taken to be as well.

\end{lemma}

\begin{proof}

We first claim that its is enough to conclude that $G$ preserves any proper subspace containing $V$. To see this, suppose $G$ preserves $K\oplus V$ for $K\subsetneq W$. Now let $K'\subset K$ be the subspace which pairs to $0$ with all of $K$. Then as $G$ is symplectic, it preserves $K'$, and since $G$ is reductive there is a complement $W'\subset K$ such that $K'\oplus W'=K$ and $G$ preserves $W'\oplus V$. As $W'$ is a complement to $K'$ it must be symplectic, and the claim is proven.

Next we claim that if  $G^{\textrm{der}}$ preserves a proper subspace containing $V$, then so does $G$. Indeed, $G=Z(G)\cdot G^{\textrm{der}}$. Since $Z(G)$ commutes with $\Sp(V)$ it must preserve $W$, and hence also $V$ as its symplectic. Now let $K\subsetneq W$ be minimal such that $V\oplus K$ is preserved by $G^{\textrm{der}}$. Then for any $z\in Z(G)$ the subspace $V\oplus zK$ is also preserved by $G^{\textrm{der}}$ and hence so is $V\oplus (K\cap zK)$. The claim follows by the minimality assumption. Therefore, we may replace $G$ by its derived subgroup and assume its semisimple and contained in $\Sp(V\oplus W)$.

We now consider the Lie algebra $L:=\textrm{Lie}(\GSp(V\oplus W))$ as a $\Sp(V)$-module via the adjoint action. We may write elements of $L$ in matrix form as $\big(\begin{smallmatrix} A & B\\C & D\end{smallmatrix}\big) $ where $A\in \fsp(V), B\in \Hom(W,V), C\in\Hom(V,W), D\in \fsp(W)$ satisfy the condition
$B=-C^\vee$ where we use the symplectic pairing to identify $V,W$ with $V^\vee,W^\vee$ respectively. 

We may therefore write $L=\big(\fsp(V)\oplus \Hom(W,V)\big)\oplus\fsp(W)$ where the two summands constitute a decomposition of $L$ into isotypic $\Sp(V)$-components, with the first summand isomorphic to a power of $V$ and the second summand a power of the trivial representation. Now consider $M:=Lie(G)\subset L$ as a $\Sp(V)$ representation, and write it as 
$$M=\fsp(V) \oplus R\oplus S$$ where $R\subset \Hom(V,W), S\subset\fsp(W)$. 
Since $L$ is a lie algebra it follows that $S$ is also a lie algebra. 
Note that as $R$ is preserved by $\Sp(V)$ it follows that $R=\Hom(V,K)$ for some $K\subset W$. 

Suppose first that $K\subsetneq W$. Note that the lie action of $S$ on $R$ is simply right composition, and therefore $S$ must preserve $K$. It follows that elements of $M$ preserve $K\oplus V$, and thus that $G$ preserves $K\oplus V$, as desired.

On the other hand, suppose that $K=W$. Then $M$ contains all of $\Hom(W,V)$. We prove now that $M$ must be all of $\fsp(V)$. Indeed, for any elements $w_1\in W, v_1\in V$ we have the element $w\ra\langle w,w_1\rangle v_1$ of $\Hom(W,V)$, which corresponds to the element $v\ra \langle v,v_1\rangle w_1$ of $\Hom(V,W)$. Applying the lie bracket for two such elements and projecting to $S$ gives us the element
$$w\ra \langle w,w_1\rangle \langle v_1,v_2\rangle w_2-\langle w,w_2\rangle \langle v_2,v_1\rangle w_1$$ of $S\cap M$. Picking $w_1=w_2, \langle v_1,v_2\rangle =\frac12$ gives us the element
$$w\ra \langle w,w_1\rangle w_1.$$ It is well known that these elementary unipotents generate all of $S$, which shows $M\supset S$ and completes the proof. 

For the final part of the claim, suppose that $V,W,G$ are defined over $F$. Now consider all the subspaces $K\subset W$ such that $V\oplus K$ is preserved by $G$. Then this set is closed under intersections and $\Aut(\bC/F)$ and therefore the minimal such $K_0$ is defined over $F$. Now let $K_1\subset K_0$ be the subspace that pairs to 0 with all of $K_0$. Then $K_1$ is also preserved by $G$, and hence there is a complementary $K_2\oplus K_1=K_0$ such that $V\oplus K_2$ is preserved by $G$. By the minimality it must be the case that $K_2=K_0$ and thus $K_0$ is symplectic. Since $K_0$ is defined over $F$, the claim is proven.

\end{proof}

\section{Proof of Theorem \ref{thm:main}}

The purpose of this section is to prove Theorem \ref{thm:main}. We shall in fact find it convenient to prove a slightly stronger theorem. To state it, we fix a counting function $\phi$ on $\cA_g$ as in section \S\ref{subsection:counting}.

\begin{thm}\label{thm:mainsstrong}
    Let $h<g$ be positive integers. Suppose that $V\subset \cA_{g+h}$ is a subvariety such that over $\bC$, the generic $g$-dimensional abelian variety is not a quotient of an abelian variety parameterized by $V$. Then a $\phi$-generic abelian variety $A$ is not a quotient of an abelian variety parameterized by $V$.
    
\end{thm}

\begin{proof}
    We assume the opposite by contradiction. Let $h\geq 0$ me minimal such that the statement is false, and pick a minimal counterexample $V\subset\cA_{g+h}$. 
    
    \begin{lemma}\label{lem:Vnotinpws}
    $V$ is not contained in a proper weakly special subvariety of $\cA_g$. 
    \end{lemma}
    \begin{proof}
        
    Indeed, if it is, then by Proposition \ref{prop:specialssplit} the abelian scheme over $V$ must be is not-simple, and thus $V$ admits a finite-to-one map to  $\cA_{g+h'}\times\cA_{h-h'}$ for some $0\leq h'<h$. In that case if we let $V_0$ be the projection of $V$ to $\cA_{g+h'}$ then the $\phi$-generic $A\in\cA_g$ is not a quotient of anything represented by $V_0$ by induction.  It follows that the $\phi$-generic $A$ is also not a quotient of anything in $V$ since such an $A$ is simple by Proposition \ref{prop:openimage}.
    \end{proof} 

    Now we set $\cA_g,\cF_g,\pi_g$ as in \S2, and we set $\cF_V:=\pi_g^{-1}(V)\cap\cF_g.$
        \begin{lemma}\label{lem:Vdimension}
     $V$ is such that $\dim V+\dim\cA_h< \dim\cA_{g+h}$
    \end{lemma}

    \begin{proof}
        Suppose not. Let $[A]\in\cA_g$ be an abelian variety 
        which is quotient of something parameterized by $V$. Then $g([A]\times \Hh_h)\cap \pi_g^{-1}(V)\neq 0$. By our dimension assumption, this is also true for a generic nearby $x\in \cA_g$. Picking $x$ to be a generic $\bC$-point of the $\Q$-variety $\cA_g$ contradicts our assumptions on $V$.
    \end{proof}

    Next, for $x\in\cF_g$ we set
    $$I_x:=\{g\in\GSp_{2g+2h}(\bC)\mid g(x\times \cF_h)\cap\cF_V \neq\emptyset\}.$$

    Note that $I\ra \cF_g$ is a definable family. Moreover, for each $g\in\GSp_{2g+2h}(\bC)$ the number of connected components of  $ g(x\times \cF_g)\cap\cF_V $ is $O_g(1)$. 

    Now, we let $[A]\in\phi^{-1}(U(\Q))$ be a $\phi$-generic point, defined over $K$, which is a quotient of an abelian variety parameterized by $V$. Then by Theorem \ref{thm:largegalois}, $[\GSp_{2g}(\hat\Z):\rho_A(G_k)]=O_g(1)$.

    Let $[B]\in V$ be such that $A$ is a quotient of $V$, let $N_A$ is the minimal degree polarized isogeny of $B$ to $A\times \cA_h$.

    \begin{lemma}\label{lem:somethingunlikely}
    For $N_A\gg 1$, There is a positive-dimensional weakly atypical subvariety $W\subset V$ which contains a conjugate of $[B]$ over $\Q(A)$. 
    \end{lemma}

    \begin{proof}
     By Theorem \ref{thm:largegalois} the field of definition of $B$ has degree $\succ_g N_A$ over $\Q$, and hence over $\Q(A)$. Write $\phi:B\ra A\times C$ for the isogeny, and by a slight abuse of notation we use $[A],[C],[B]$ to refer to the representatives in the fundamental domains $\cF_g,\cF_h,\cF_{g+h}$. By Corollary \ref{cor: isogenyheightbound} it follows that there exists an $M\in I_{[A]}(\Q)$ of height $H(M)\prec_g N_A$ such that $M([A]\times[C]) = [B]$. In fact, we obtain such an element $M_{B'}$ for each Galois conjugate $B'$ of $B$.

     Suppose the number of distinct $M_{B'}$ we obtain in this way is $N_A^{o(1)}$. Then for some such element $M$ we must have that  $M\cdot([A]\times \cF_h) \cap \cF_V$ contains $\succ_g N_A$ distinct conjugates of $B$. Since this is a definable family as $M$ varies, it follows that for $N_A\gg 1$ that it must consist of at least one positive dimensional component $U$ containing a conjugate of $B$ (in fact, many such conjugates). By Lemma \ref{lem:Vdimension} we see that $U$ is an unlikely intersection. By Theorem \ref{thm:AS} it follows that $U$ is contained in a weakly atypical subvariety of $V$. 

     Thus we assume that the number of distinct $M_{B'}$ is $\prec_g N_A$. By Theorem \ref{pointcount2}, we conclude that for $N_A\gg 1$,  $I_{[A]}$ contains a semi-algebraic curve $C_0$ containing some $M_{B'}$ for $B'$ a conjugate of $B$. Let $C$ be the complex algebraic curve containing $C_0$. We now have two cases.

    Suppose that for some conjugate $B'$ we have $M_{B'}\cdot([A]\times \cF_h) \cap \cF_V$ is positive dimensional, and let $U$ denote an analytic component of such an intersection. We then conclude as above.
    
    Alternatively, we see that for a generic $c\in C$ the variety $c\cdot ([A]\times \cF_h) \cap \cF_V$ is finite. Since this is a definable family of (0-dimensional )varieties as $c$ varies, it follows that thse intersections are of size $O_g(1)$. Hence the intersections must vary with $c\in C$ and thus 
    $C\cdot ([A]\times \cF_h) \cap \cF_V$ is positive dimensional. Let $U$ be an analytic component of the intersection. Again, by Lemma \ref{lem:Vdimension} we know that $U$ is an unlikely intersection, so we conclude this case in the same way.

    \end{proof}
    
    We now apply our Geometric Zilber-Pink result \ref{thm:GZP} to conclude that a conjugate of $B$ is contained in a proper subvariety $V_1\subset V$ which contains all the positive-dimensional weakly atypical subvarieties of $V$. It follows that for $\phi$-generic $A$ with $N_A\gg 1$, if $A$ is a quotient of something represented by $V$ then $A$ is a quotient of something represented by $V_1$. By our minimality assumption on $V$,  we conclude that for  a $\phi$-generic $A$ which is a quotient of something represented by $V$, the corresponding $N_A$ is bounded.

    Now, for a fixed $N$ the set of $A$ with $N_A=N$ lies in the projection of $T_N\cdot V\cap (\cA_g\times \cA_h)$ to $\cA_g$ for $T_N$ the Hecke-correspondence for polarized isogenies of degrees $N$. This must be a proper subvariety, by the assumption on $V$, and hence has density 0. This yields the desired contradiction.

\end{proof}

\end{document}